\documentclass[12pt]{amsart}
\usepackage{amsmath, amssymb, amsthm, latexsym}
\usepackage{amssymb}
\usepackage{latexsym}
\usepackage[center]{caption}
\usepackage[OT2,T1]{fontenc}
\DeclareSymbolFont{cyrletters}{OT2}{wncyr}{m}{n}
\DeclareMathSymbol{\Sha}{\mathalpha}{cyrletters}{"58}
\usepackage{tikz}
\newcounter{braid}
\newcounter{strands}
\pgfkeyssetvalue{/tikz/braid height}{1cm}
\pgfkeyssetvalue{/tikz/braid width}{1cm}
\pgfkeyssetvalue{/tikz/braid start}{(0,0)}
\pgfkeyssetvalue{/tikz/braid colour}{black}
\pgfkeys{/tikz/strands/.code={\setcounter{strands}{#1}}}

\makeatletter
\def\cross{%
  \@ifnextchar^{\message{Got sup}\cross@sup}{\cross@sub}}

\def\cross@sup^#1_#2{\render@cross{#2}{#1}}

\def\cross@sub_#1{\@ifnextchar^{\cross@@sub{#1}}{\render@cross{#1}{1}}}

\def\cross@@sub#1^#2{\render@cross{#1}{#2}}

\def\render@cross#1#2{
  \def\strand{#1}
  \def\crossing{#2}
  \pgfmathsetmacro{\cross@y}{-\value{braid}*\braid@h}
  \pgfmathtruncatemacro{\nextstrand}{#1+1}
  \foreach \thread in {1,...,\value{strands}}
  {
    \pgfmathsetmacro{\strand@x}{\thread * \braid@w}
    \ifnum\thread=\strand
    \pgfmathsetmacro{\over@x}{\strand * \braid@w + .5*(1 - \crossing) * \braid@w}
    \pgfmathsetmacro{\under@x}{\strand * \braid@w + .5*(1 + \crossing) * \braid@w}
    \draw[braid] \pgfkeysvalueof{/tikz/braid start} +(\under@x pt,\cross@y pt) to[out=-90,in=90] +(\over@x pt,\cross@y pt -\braid@h);
    \draw[braid] \pgfkeysvalueof{/tikz/braid start} +(\over@x pt,\cross@y pt) to[out=-90,in=90] +(\under@x pt,\cross@y pt -\braid@h);
    \else
    \ifnum\thread=\nextstrand
    \else
     \draw[braid] \pgfkeysvalueof{/tikz/braid start} ++(\strand@x pt,\cross@y pt) -- ++(0,-\braid@h);
    \fi
   \fi
  }
  \stepcounter{braid}
}

\tikzset{braid/.style={double=\pgfkeysvalueof{/tikz/braid colour},double distance=1pt,line width=2pt,white}}

\newcommand{\braid}[2][]{%
  \begingroup
  \pgfkeys{/tikz/strands=2}
  \tikzset{#1}
  \pgfkeysgetvalue{/tikz/braid width}{\braid@w}
  \pgfkeysgetvalue{/tikz/braid height}{\braid@h}
  \setcounter{braid}{0}
  \let\sigma=\cross
  #2
  \endgroup
}
\makeatother

\input xypic

\swapnumbers
\newtheorem{theorem}{Theorem}

\newtheorem{lemma}[theorem]{Lemma}
\newtheorem{corollary}[theorem]{Corollary}


\def\Z{\mathbb{Z}}

\def\C{\mathbb{C}}

\def\Q{\mathbb{Q}}

\def\R{\mathbb{R}}
\def\C{\mathbb{C}}

\def\G{\mathbb{G}}

\def\qed{\hfill$\square$\medskip}

\def\Zpk{\mathbb{Z}/p^{k}}
\def\Zpk1{\mathbb{Z}/p^{k-1}}

\newcommand{\rref}[1]{(\ref{#1})}

\newcommand{\beg}[2]{\begin{equation}\label{#1}#2\end{equation}}
\def\r{\rightarrow}

\def\Zf{\mathbf{Z}}

\def\sl2{\widetilde{SL_{2}(\Z)}}

\title[$\Sha$ and a homotopy limit problem]{On the arithmetic of elliptic curves and a
homotopy limit problem}
\author{Igor Kriz}
\thanks{Igor Kriz was supported by NSF grant DMS 1102614 and by a grant from the Simons Foundation}

\begin{document}

\maketitle

\begin{abstract}
In this note, I study a comparison map between a motivic and \'{e}tale cohomology group of an elliptic curve over
$\Q$ just outside the range of Voevodsky's isomorphism theorem. I show that the property of an appropriate version of
the map being an isomorphism is equivalent to certain arithmetical properties of the elliptic curve.
\end{abstract}

\vspace{3mm}

\section{Introduction}\label{s1}

The most well known theorem of motivic homotopy theory is Voevodsky's proof of the Beilinson-Lichtenbaum 
and Bloch-Kato
conjectures \cite{voev}.
In one form (\cite{voev}, Theorem 6.17), this result states that for a pointed smooth simplicial scheme $X$,
the natural homomorphism
\beg{evoev}{\widetilde{H}^{p}_{Mot}(X,\Z/\ell(q))\r \widetilde{H}^{p}_{\text{\em\'{e}t}}(X,\Z/\ell(q))}
is an isomorphism for $p\leq q$ and a monomorphism for $p=q+1$. 

\vspace{3mm}
The purpose of the present note
is to study the map \rref{evoev} when $X$ is an elliptic curve over $\Q$, $p=2$, $q=1$. In this case,
we know from Voevodsky's theorem that \rref{evoev} is a monomorphism.

\begin{theorem}\label{tm1}
Let $X=E$ be an elliptic curve defined over $\Q$. Then the canonical homomorphism
$$\widetilde{H}^{2}_{Mot}(E,\mathbf{Z}_\ell(1))\r \widetilde{H}^{2}_{\text{\em\'{e}t}}(E,\mathbf{Z}_\ell(1))$$
where $\mathbf{Z}_\ell(1)$ denotes the homotopy limit of $\Z/\ell^k(1)$ in the category of motives (resp. \'{e}tale
motives)
always has an uncountable cokernel.
\end{theorem}

The situation changes, however, if we work with finite models. For a large enough set $S$ of primes in $\Z$, 
an elliptic curve $E$ over $\Q$ has a smooth projective
model over $\Z[S^{-1}]$, which we will denote by $E[S^{-1}]$.

\begin{theorem}\label{t0}
Let $X=E$ be an elliptic curve defined over $\Q$. Then the canonical homomorphism
$$\lim_{\begin{array}[t]{c}\r\\[-.5ex]S\end{array}}\widetilde{H}^{2}_{Mot}(E[S^{-1}],\mathbf{Z}_\ell(1))\r 
\lim_{\begin{array}[t]{c}\r\\[-.5ex]S\end{array}}\widetilde{H}^{2}_{\text{\em\'{e}t}}(E[S^{-1}]
,\mathbf{Z}_\ell(1))$$
is an isomorphism if and only if $\Sha(E)_{(\ell)}$ is finite and $rank_\Q(E)>0$. 
\end{theorem}

\vspace{3mm}
\noindent
{\bf Remark:} Both direct limits in the statement of the Theorem are 
in fact eventually constant.

\vspace{3mm}
Here 
$$\Sha(E)=\bigcap_\nu Ker(H^1(\Q,E)\r H^1(\Q_\nu, E_\nu))$$
(where the intersection is taken over all completions of $\Q$)
is the {\em Tate-Shafarevich group}, the finiteness of which (even at one prime) is equivalent to 
the vanishing of the discrepancy between the rank of the group of rational points of $E$ and its computable 
estimate (see, for example, \cite{milne} for an introduction).

\vspace{3mm}
We see easily (as reviewed in the next section) that for $p=2$, $q=1$, \rref{evoev} is never an isomorphism 
for $X=S^0$. Therefore, it would never be an isomorphism for an elliptic curve if we took unreduced 
instead of reduced cohomology. It is worthwhile noting that philosophically speaking, by taking reduced cohomology, the
weight of the motive in question increases by $1$. If it increased by $2$, we would be back in the range of 
Voevodsky's isomorphism theorem. Consequently, we are investigating a cohomology group which is really ``just over the isomorphism line''.

\vspace{3mm}
Let
$T_\ell(E)$ be the $\ell$-adic Tate module of $E$, i.e. the inverse limit of its $\ell^n$-torsion .
At some point in the proof, Theorem \ref{t0} is rephrased as 
the following
statement in pure arithmetic:

\begin{theorem}\label{t00}
Let $E$ be an elliptic curve over $\mathbb{Q}$. Let $\ell$ be a prime. Then for a sufficiently large
finite set of primes $S$ in $\Z$, the Kummer map
$$E({\Z[S^{-1}]})\otimes \Z_\ell\r H^1_{\text{\em\'{e}t}}(\Z[S^{-1}],T_\ell(E))$$
is an isomorphism if and only if $\Sha(E)_{(\ell)}$ is finite and $rank_\Q(E)>0$. 
\end{theorem}

A reader interested only in arithmetic and not motivic cohomology can consider this statement only, and
skip directly to Section \ref{shard}. To the author, the motivic statement was the original motivation, which led
to the observation. The author thanks J.Nekov\'a\v{r}, C.Skinner and C.Weibel for discussions and for 
pointing out mistakes in eariler statements of this simple but tricky result, and for helping to correct them.

\vspace{3mm}
The present note is organized as follows: We review some notation and fix some definitions in the next section,
and we give a more definitive statement of Theorem \ref{t0}. In Section \ref{s4}, we prove the easier of the 
two main implications of the theorem. In Section \ref{shard}, we prove the harder implication, and also
Theorem \ref{t00}. Finally, in Section \ref{s5}, we give an example where 
the statement of the harder implication can be proved by more elementary means.

\vspace{3mm}

\noindent

\vspace{5mm}
\section{Basic Definitions and the Main Theorem}
\label{s3}

Let $E$ be a smooth projective variety over a Noetherian scheme $Z$. We will mostly be interested in 
the case where
\beg{eass1}{\parbox{3.5in}{$Z=Spec(\Q)$ or $Z=Spec(\Z[S^{-1}])$ where $S$ is some finite set of primes.}}
Let us begin with reviewing some notation. The Kummer short exact sequence of \'{e}tale sheaves 
\beg{egm}{\diagram 0\rto & \Z/\ell^k(1)\rto& \mathbb{G}_m\rto^{\ell^k}& \mathbb{G}_m\rto & 0
\enddiagram}
gives rise to a cofibration sequence in the derived category of \'{e}tale sheaves 
$$\diagram
\mathbb{G}_m\rto^{\ell^k} & \mathbb{G}_m\rto^(.35){\phi_k} &\Z/\ell^k(1)[1].
\enddiagram$$
We then have the canonical homomorphism
$$\phi_{k*}:H^{1}_{\text{\em\'{e}t}}(E,\mathbb{G}_m)\r H^{2}_{\text{\em\'{e}t}}(E,\Z/\ell^k(1)).$$
In this paper, we will make use of the derived categories of motives and the derived category of \'{e}tale 
motives $\mathbf{DM}^-_{Nis}$, $\mathbf{DM}^-_{\text{\em\'{e}t}}$ (\cite{vss,mvv}). Constant sheaves,
$\mathbb{G}_m$, $\mathbb{Z}/\ell^k(1)$ are examples of \'{e}tale sheaves with transfers, thereby 
defining objects of $\mathbf{DM}^-_{Nis}$, $\mathbf{DM}^-_{\text{\em\'{e}t}}$. We will denote the 
corresponding objects of those categories by the same symbols. Smooth schemes over $Z$
have well defined cohomology with coefficients in an object of
$\mathbf{DM}^-_{Nis}$ or $\mathbf{DM}^-_{\text{\em\'{e}t}}$. If the object of 
$\mathbf{DM}^-_{Nis}$ or $\mathbf{DM}^-_{\text{\em\'{e}t}}$
comes from a homotopy invariant Nisnevich
resp. \'{e}tale sheaf with transfers, the cohomology with coefficients in the motive is the same 
as the corresponding Nisnevich (resp. \'{e}tale) cohomology. Moreover, Nisnevich
cohomology of smooth schemes with coefficients in homotopy invariant
sheaves is the same as Zariski cohomology (\cite{mvv}, Proposition 13.9). We will refer to the latter
simply as {\em motivic cohomology}. This justifies our identification of symbols, since 
we are solely interested in cohomology. We will decorate motivic resp.
\'{e}tale cohomology as $H_{Mot}$, $H_{\text{\em{\'{e}t}}}$, thus eliminating the need to distinguish
notations on the level of coefficients.

\vspace{3mm}
Next, in $\mathbf{DM}^-_{Nis}$, $\mathbf{DM}^-_{\text{\em\'{e}t}}$, we shall write
\beg{elet}{\begin{array}{c}\displaystyle\Zf_\ell=\operatornamewithlimits{holim}_\leftarrow \Z/\ell^k\\
\displaystyle\Zf_\ell(1)=\operatornamewithlimits{holim}_\leftarrow \Z/\ell^k(1).
\end{array}}
It is important to note that these are {\em not} the same objects as $\Z_\ell$, 
$\Z_\ell(1)$, which mean the Nisnevich or \'{e}tale constant sheaf and its tensor with
$\Z(1)$ respectively (or the associated Nisnevich or \'{e}tale motive). 
For example, for $Z=Spec(k)$ where $k$ is a field, by Theorem 4.1 of \cite{mvv},
$$H^1(Spec(k),\Z_\ell(1))=k^\times\otimes_\Z\Z_\ell,$$
which is in general not equal to 
$$H^1(Spec(k),\Zf_\ell(1))=\lim_\leftarrow (k^\times/(k^\times)^{\ell^m}).$$
We have the usual $\lim{}^1$ exact sequence
\beg{elim1}{0\r\lim{}^1 H^{i-1}_{\text{\em\'{e}t}}(E,\Z/{\ell^k}(1))
\r H^{i}_{\text{\em\'{e}t}}(E,\Zf_\ell(1))\r \lim_\leftarrow H^{i}_{\text{\em\'{e}t}}(E,\Z/\ell^k(1))\r 0.
}
There is also a similar short exact sequence for the motivic groups. \'{E}tale cohomology groups
with coefficients in $\mathbf{Z}_\ell(n)$ were first introduced by U.Jannsen \cite{jan}.

We have a canonical diagram
\beg{ed1}{\diagram
H^{1}_{\text{\em\'{e}t}}(E,\mathbb{G}_m)
\drto\rrto^{\Phi} && H^{2}_{\text{\em\'{e}t}}(E,\Zf_\ell(1))\\
&H^{1}_{\text{\em\'{e}t}}(E,\mathbb{G}_m)\otimes \Z_\ell.\urto &
\enddiagram
}

\vspace{3mm}
\begin{lemma}\label{ltwist1}
We have the following isomorphisms both in $\mathbf{DM}^-_{Nis}$ and $\mathbf{DM}^-_{\text{\em\'{e}t}}$:
\beg{ed2}{\diagram
\Z(1)\otimes \Zf_\ell \rto^(.6)\simeq &\Zf_\ell(1).
\enddiagram
}
\end{lemma}

\begin{proof}
In the category of derived motives, $\Z(1)=\mathbb{G}_m[-1]$
is an invertible and hence strongly dualizable object with dual $\Z(-1)$, so we have
$$
\begin{array}{l}\displaystyle
\Z(1)\otimes \Zf_\ell=Hom(\Z(-1),\Zf_\ell)=\\[2ex]
\displaystyle
\operatornamewithlimits{holim}_\leftarrow 
Hom(\Z(-1),\Z/\ell^m)=\operatornamewithlimits{holim}_\leftarrow \Z/\ell^m(1)=\Zf_\ell(1),
\end{array}
$$ 
as claimed.
\end{proof}

But also a smooth projective variety is strongly dualizable in the stable motivic homotopy category,
and therefore its cohomology is equal to the homology of its dual. It follows that in 
the following comparison diagram, the top row (with notation
analogous to the \'{e}tale case) is an isomorphism in the case when $E$ is an elliptic curve, and we have
\rref{eass1}:
\beg{ed3}{\diagram
H^{1}_{Mot}(E,\mathbb{G}_m)\otimes\Z_\ell\rto^{\cong}_(.55){\Phi_{Mot}}
\dto_{\cong}^{\rho\otimes \Z_\ell}&H^{2}_{Mot}(E,\Zf_\ell(1))
\dto^\rho\\
H^{1}_{\text{\em\'{e}t}}(E,\mathbb{G}_m)\otimes\Z_\ell \rto^\Phi & H^{2}_{\text{\em\'{e}t}}(E,
\Zf_\ell(1)).
\enddiagram
}
(To see that $\Phi_{Mot}$ is an isomorphism in \rref{ed3}, note that the group of rational
points $E(\Q)$ is a finitely generated
abelian group. We have
$$H^2_{Mot}(E,\Z/\ell^m(1))=E(\Q)/\ell^m$$
by the Kummer exact sequence, while 
$H^1_{Mot}(E,\Z/\ell^m(1))$ is the $\ell^m$-torsion subgroup of $E(\Q)$, which is finite
and hence the $\lim{}^1$ term vanishes in the motivic analogue of \rref{elim1}.)

On the other hand, the realization map
$$\rho:H^{1}_{Mot}(E,\mathbb{G}_m)\r H^{1}_{\text{\em\'{e}t}}(E,\mathbb{G}_m)$$
is well known to be an isomorphism (a version of Hilbert 90 theorem, see e.g. \cite{etale}). Therefore,
the left column of diagram \rref{ed3} is an isomorphism.

\vspace{3mm}
We must discuss another point. For a scheme $X$, one defines the {\em Brauer group}
$$Br(X)=H^2_{\text{\em\'{e}t}}(X,\mathbb{G}_m).$$
Define also
$$T_\ell Br(X)=\lim_\leftarrow {}_{\ell^k}Br(X)$$
where ${}_{n}Br(X)$ is the $n$-torsion in $Br(X)$ (i.e. the subgroup of elements $x$ where
$nx=0$). One writes $Br(R)$ instead of $Br(Spec(R))$.

\vspace{3mm}
As stated, there is no chance that the map $\rho$ (or
$\Phi$) of diagram \rref{ed3} would be an isomorphism. Let us consider the case of
$Spec(R)$ where $R$ is a number field or $\Z[S^{-1}]$. Then 
$$H^{1}_{Mot}(Spec(R),\G_m)=0,$$
and hence 
\beg{emot0}{H^{2}_{Mot}(Spec(R),\Zf_\ell(1))=0,}
whereas by \rref{elim1}, we have a short exact sequence
$$\begin{array}{c}0\r \lim{}^1 H^{1}_{\text{\em\'{e}t}}(Spec(R),\Z/\ell^k(1))\r 
H^{2}_{\text{\em\'{e}t}}(Spec(R),\Zf_\ell(1))\\
\r \displaystyle\lim_\leftarrow H^{2}_{\text{\em\'{e}t}}(
Spec(R),\Z/\ell^k(1))\r 0,\end{array}$$
or, using \rref{egm},
$$0\r\lim{}^1 R^\times/R^{\times\ell^k}\r H^{2}_{\text{\em\'{e}t}}(Spec(R),\Zf_\ell(1))\r \lim_\leftarrow
{}_{\ell^R}Br(R)\r 0.$$
The first term is clearly $0$ (since the maps are onto), so we get
\beg{emot1}{H^{2}_{\text{\em\'{e}t}}(R,\Zf_\ell(1))\cong \lim_\leftarrow {}_{\ell^k}Br(R).
}
The right hand side of \rref{emot1} is nonzero
by class field theory. However, if $E$ has a point
over $k$, the map $\rho$ from \rref{emot0} to \rref{emot1} is a retract of
the map $\rho$ in \rref{ed3}, so the map $\rho$ cannot be an isomorphism. 
As customary, we will denote by $\widetilde{H}$ the kernel of either row of the
diagram \rref{ed3} to $Z$ induced by a $Z$-point in $E$, and call this summand the
corresponding {\em reduced cohomology group}.

\vspace{3mm}

\vspace{3mm}
Let $\ell=2,3,5,\dots$ be a  prime and let $E$ be an elliptic curve 
defined over $\Q$. Denote by $S$
the (finite)
set of all primes in $\Z$ dividing the conductor of $E$.
Note that by the criterion of N\'{e}ron-Ogg-Shafarevich, $T_\ell(E)$ is unramified at all primes $p\notin S$, 
and the elliptic curve $E$ has a smooth projective model over $\Z[S^{-1}]$, which we will denote by
$E(\Z[S^{-1}])$.

\vspace{3mm}
\begin{theorem}
\label{t1}

The following are equivalent:

(a) $\Sha(E/\Q)\otimes\Z_{(\ell)}$ is
finite and $rank_\Q(E)>0$.

(b)
The realization map of diagram \rref{ed3}
\beg{ereg}{\rho:\widetilde{H}^{2}_{Mot}(E(\Z[S^{-1}]), \Zf_\ell(1))\r 
\widetilde{H}^{2}_{\text{\em\'{e}t}}(E(\Z[S^{-1}]),\Zf_\ell(1))}
is an isomorphism.

(c) The map $\rho$ of \rref{ereg}  is onto.

(d) The map
\beg{ereg1}{\Phi:\widetilde{H}^{1}_{\text{\em\'{e}t}}(E(\Z[S^{-1}]),\mathbb{G}_m)\otimes \Z_\ell
\r \widetilde{H}^{2}_{\text{\em\'{e}t}}(E(\Z[S^{-1}]),\Zf_\ell(1))
}
is an isomorphism.

(e) The map $\Phi$ of \rref{ereg1} is onto. 

(f) The map
\beg{ebrauertate}{T_\ell Br(E(\Z[S^{-1}]))\r T_\ell Br(\Z[S^{-1}])
}
induced by the inclusion of $0\in E$ is an isomorphism.

\end{theorem}

\vspace{3mm}

\vspace{5mm}

\section{Proof of the Main Theorem - The Easy Implication}
\label{s4}

Consider diagram \rref{ed3} and the fact that the maps $\Phi$, $\rho$
when reduced $\mod \ell^k$ become (by definition) isomorphisms.
Since the source of $\Phi$ is a finitely generated $\Z_\ell$-module,
it has no infinite $\ell$-divisibility, which implies that $\Phi$ (and hence
$\rho$) is injective. Therefore, we know that (b), (c), (d) and (e) of the statement
are equivalent.

\vspace{3mm}
Let us prove that (b) implies the first statement of (a), i.e. that $\Sha(E/\Q)\otimes\Z_{(\ell)}$ is finite.
We follow \cite{milne}, Chapter IV.2. Consider the diagram
\beg{edselmer}{\diagram
E(\Q)/\ell E(\Q)\rto^{i_1} & S^{(\ell)}(E/\Q) \rto^{j_1} & H^1(\Q,{}_\ell E)\\
E(\Q)/\ell^2 E(\Q)\uto^{\pi_1}\rto^{i_2} &
S^{(\ell^2)}(E/\Q)\uto^{\alpha_1}\rto^{j_2} &
H^1(\Q,{}_{\ell^2}E)\uto^{\gamma_1}\\
E(\Q)/\ell^3 E(\Q)\uto^{\pi_2}\rto^{i_3} &
S^{(\ell^3)}(E/\Q)\uto^{\alpha_2}\rto^{j_3} &
H^1(\Q,{}_{\ell^3}E)\uto^{\gamma_2}\\
\vdots \uto^{\pi_3} & \vdots\uto^{\alpha_3} &
\vdots. \uto^{\gamma_3}
\enddiagram}
Here, as usual, ${}_nE$ denotes the $n$-torsion in $E$, and $S^{(n)}$ denotes the Selmer group,
i.e. the kernel of the map
$$H^1(\Q,{}_nE)\r\prod_p H^1(\Q_p,E).$$
The maps $i_n$, $j_n$ are inclusions, the $\pi_n$ are projections
(hence onto), and the other vertical maps are induced by projections.
Since $\Sha(E/\Q)_{\ell^n}$ is a quotient of the finite group
$S^{(\ell^n)}(E/\Q)$, it is finite, so finiteness of $\Sha(E/\Q)\otimes \Z_{(\ell)}$
is equivalent to the absence of infinitely $\ell$-divisible non-zero elements in
$\Sha(E/\Q)$. This, in turn, is equivalent to asserting that
\beg{e1}{\text{$i_1:E(\Q)/\ell E(\Q)\r \bigcap_{n} Im (\alpha_1\alpha_2\dots \alpha_n)$
is onto.}
}
Clearly, \rref{e1} follows from 
\beg{e2}{\text{$j_1i_1:E(\Q)/\ell E(\Q)\r \bigcap_{n} Im (\gamma_1\gamma_2\dots \gamma_n)$
is onto.}
}
We shall prove \rref{e2}. 
We have
$$H^2_{Mot}(E,\Z(1))\cong \Z \oplus E(\Q)$$
where the first summand corresponds to the degree. More precisely, there is a short exact
sequence of the form
\beg{edeg}{\diagram
0\rto&H^2_{Mot}(E,\Z(1))_0
\rto &
H^2_{Mot}(E,\Z(1))
\rto^(.7){deg} & \Z\rto & 0,
\enddiagram
}
and there is a canonical isomorphism
$$H^2_{Mot}(E,\Z(1))_0\cong E(\Q).$$
We shall also be interested in the $\ell$-adic version of \rref{edeg}:
$$
\diagram
0\rto&H^2_{Mot}(E,\Zf_\ell(1))_0
\rto &
H^2_{Mot}(E,\Zf_\ell(1))
\rto^(.7){deg} & \Z_\ell\rto & 0.
\enddiagram
$$
Now consider the diagram
{\protect\beg{e2a}{\resizebox{\textwidth}{!}{\protect{\diagram
H^2(E(\Z[S^{-1}]),\Z/\ell\Z(1))\rto^{\overline{r}}&
\widetilde{H}^{2}_{\text{{\em\'{e}t}}}(E(\Z[S^{-1}]),\Z/\ell\Z(1))_0\rto^{\overline{u}}& H^1_{\text{{\em\'{e}t}}}
(\Z[S^{-1}],{}_\ell E)\\
E(\Q)/\ell E(\Q)\uto_{\subseteq}\rto_{r^\prime} &
\widetilde{H}^{2}_{\text{{\em\'{e}t}}}(E(\Z[S^{-1}],\Zf_\ell(1))_0/(\ell)\uto_q\rto_{u^\prime} & 
H^1_{\text{{\em\'{e}t}}}(\Z[S^{-1}],T_\ell(E))/(\ell)\uto^{\subseteq}_{s}\\
H^{2}_{Mot}(E(\Q),\Zf_\ell(1))_0\uto|>>\tip^{\pi^\prime}\rto^{r}
&\widetilde{H}^{2}_{\text{{\em\'{e}t}}}(E(\Z[S^{-1}]),\Zf_\ell(1))_0\uto^{\gamma^\prime}_{\cong}\rto^u &
H^1_{\text{{\em\'{e}t}}}(\Z[S^{-1}],T_\ell(E)).\uto|>>\tip_\gamma
\enddiagram}}
}}
To explain this, first note that we have
$$H^{2}_{Mot}(E,\Zf_\ell(1))_0/(\ell)=E(\Q)/\ell E(\Q).$$
Next, the \'{e}tale group with subscript $0$ is defined in analogy with the
corresponding motivic group, i.e. as the kernel of the degree map. 
The maps $\pi^\prime$, $\gamma^\prime$, $\gamma$ are reductions
$\mod\ell$, so they are onto. The map $r$ is \'{e}tale realization, and is
an isomorphism by our assumption (b). 
The maps $s,q$ are inclusions coming from the Bockstein long exact sequence associated
with 
$$\diagram
0 \rto & T_\ell(E)\rto^\ell & T_\ell(E)\rto & {}_\ell E\rto & 0
\enddiagram$$
where $T_\ell(E)$ is the Tate module, i.e. 
$$\lim_\leftarrow {}_{\ell^k}E,$$
which is non-canonically isomorphic to $\Z_\ell^2$.
Now the map $u$ comes from the Hochschild-Serre spectral sequence 
\beg{e3}{
E_2=H^p_{\text{{\em\'{e}t}}}(\Z[S^{-1}],H^{q}_{\text{{\em\'{e}t}}}(E(K),\Zf_\ell(1)))
\Rightarrow H^{p+q}_{\text{{\em\'{e}t}}}(E(\Z[S^{-1}]),\Zf_\ell(1))
}
where $K$ is the maximal extension of $\Q$ over which all the primes outside of $S$ are unramified.
We have
$$T_\ell(E)=H^{1}_{\text{{\em\'{e}t}}}(E(K),\Zf_\ell(1)),$$
so the $p=q=1$ term is the target of $u$. Note (using purity) that the
$p=0,q=2$ term is
$$H^{0}_{\text{{\em\'{e}t}}}(\Z[S^{-1}],\Zf_\ell)=\Z_\ell,$$
and the edge map in $p+q=2$ is the degree map.
Note also that the $p=2$, $q=0$ term is
$$H^{2}_{\text{{\em\'{e}t}}}(\Z[S^{-1}],\Zf_\ell(1));$$
the canonical map of the right hand side to $ H^{2}_{\text{\em\'{e}t}}(E(\Z[S^{-1}]),\Z_\ell(1)) $ is
the edge map. Therefore, since $E$ contains a point over $\Z[S^{-1}]$, 
the projection given by the spectral sequence
$${H}^{2}_{\text{{\em\'{e}t}}}(E(\Z[S^{-1}]),\Zf_\ell(1))_0\r
H^1_{\text{{\em\'{e}t}}}(\Z[S^{-1}],T_\ell(E))$$
factors through an injection
$$u:\widetilde{H}^{2}_{\text{{\em\'{e}t}}}(E(\Z[S^{-1}]),\Zf_\ell(1))_0\r H^1_{\text{{\em\'{e}t}}}
(\Z[S^{-1}],T_\ell(E)).$$
The map $\overline{u}$ is defined as the corresponding map for
the analogous spectral sequence $\overline{E}^{r}_{pq}$
with coefficients reduced $\mod \ell$. 

At this point, let us first assume that $\ell\neq 2$.
Then $u$ is onto since the only possible differential of \rref{e3}
originating at $p=q=1$ has target
$$H^{3}_{\text{{\em\'{e}t}}}(\Z[S^{-1}],H^{0}_{\text{{\em\'{e}t}}}(E(K),
\Zf_\ell(1)))=0.$$
Next, observe that by the Bockstein spectral sequence,
\beg{e4}{Im(s\gamma)=\bigcap_n \gamma_1\dots\gamma_n,}
while
\beg{e5}{su^\prime r^\prime=j_1i_1.
}
In effect, to prove \rref{e5}, let us spell out the definition of $j_1i_1$:
Take $x\in E(\Q)$, and set
\beg{e6}{y=\sqrt[\ell]{x}\in E.
}
Then define a $1$-cocycle on $Gal(\Q)$ by setting
\beg{e7}{g\mapsto \frac{g(y)}{y}.
}
Now the analogue $\overline{E}^{r}_{p,q}$ of \rref{e3} with coefficients
reduced $\mod \ell$ has a motivic analogue ${}_{Mot}\overline{E}^{r}_{p,q}$
(although we do not know whether it converges). Nevertheless, we have a realization
map of exact couples, and hence spectral sequences
\beg{e8}{{}_{Mot}\overline{E}^{r}_{p,q}\r \overline{E}^{r}_{p,q}.
}
On $r=2$, $p=q=1$, and $r=2$, $p=0$, $q=2$, the map \rref{e8} is an isomorphism.
Now on the level of $\mod \ell$ motivic cohomology, the definition corresponding
to \rref{e6} and \rref{e7} is equal to $\overline{u}\overline{r}$ by the definition
of the exact couple which produces ${}_{Mot}\overline{E}^{r}_{p,q}$, which proves
\rref{e5}.

Now since $u,\gamma$ are onto and $r$ is an isomorphism, $u^\prime r^\prime$ is
onto, and so is $\gamma$, so
$$Im(s\gamma)=Im(su^\prime s^\prime)=Im(s).$$
Therefore, \rref{e4} and \rref{e5} imply \rref{e2}.

Now let us treat the case $\ell=2$. We see that all that remains to show is that $u$ is
onto, which follows from the following Lemma.

\vspace{3mm}

\begin{lemma}
\label{l1}
In the Hochschild-Serre spectral sequence \rref{e3}, we have
\beg{e+}{\diagram
H^1_{\text{{\em\'{e}t}}}(\Z[S^{-1}],
H^{1}_{\text{\em\'{e}t}}(E(K),\Zf_2(1)))=H^1_{\text{{\em\'{e}t}}}(\Z[S^{-1}],T_2(E))\dto^{d_2=0}\\
H^3_{\text{{\em\'{e}t}}}(\Z[S^{-1}],H^{0}_{\text{\em\'{e}t}}(E(K),\Zf_2(1)))=H^3_{\text{{\em\'{e}t}}}
(\Z[S^{-1}],\Zf_2(1)).
\enddiagram
}
\end{lemma}

\begin{proof}
We have a restriction comparison diagram
$$\diagram
H^1_{\text{{\em\'{e}t}}}(\Z[S^{-1}],H^{1}_{\text{\em\'{e}t}}(E(K),\Zf_2(1)))\dto\rto^{d_2}&
H^3_{\text{{\em\'{e}t}}}(\Z[S^{-1}],H^{0}_{\text{\em\'{e}t}}(E(K),\Zf_2(1)))\dto^\cong\\
H^1(\R,H^{1}_{\text{\em\'{e}t}}(\overline{E},\Zf_2(1)))\rto^{d_2}
& H^3(\R,H^{0}_{\text{\em\'{e}t}}(\overline{E},\Zf_2(1)))
\enddiagram$$
where the right hand column is an isomorphism by Theorem B, p. 108 of \cite{serre}.
Thus, we may replace $\Q$ by $\R$ in \rref{e+}.

Then, however, we are dealing with a spectral sequence isomorphic to the
Borel cohomology spectral sequence for the $\Z/2\Z$-action by complex conjugation
on a (complexified) real elliptic curve, i.e. $E_\C=\C^\times/q^\Z$, $q\in\R$,
$0<q<1$. We see then that topologically $\Z/2\Z$-equivariantly, we have
\beg{e++}{E\cong S^1\times S^\alpha
}
where $\alpha$ is the sign representation and $S^V$ is the one point compactification of
a representation $V$. But stably (i.e. after taking suspension spectra), \rref{e++} splits as
$$S^0\vee S^1\vee S^\alpha\vee S^{1+\alpha},$$
and hence the Borel cohomology spectral sequence collapses.
\end{proof}

\section{The Hard Implication}\label{shard}

\vspace{3mm}
Thus, we have shown that (b) implies the first statement of (a) To complete the proof, we will show that
(a) implies (e), and that (e) together with the finiteness of $\Sha(E/\Q)\otimes\Z_{(\ell)}$ implies $rank_\Q(E)>0$.

We shall make use of the following

\begin{lemma}
\label{l2}
Let $E$ be an elliptic curve defined over $\Q$. Let $\ell, p$ be primes. Then 
\beg{evan}{H^{1}(\Q_p,T_\ell(E))\otimes_{\Z_\ell}\Q_\ell\cong
\left\{
\begin{array}{ll}
0 & \text{if $p\neq \ell$}\\
\Q_p\oplus \Q_p &\text{if $p=\ell$}
\end{array}
\right.
}
\end{lemma}

\begin{proof}
Let 
$$V_\ell(E)=T_\ell(E)\otimes_{\Z_\ell}\Q_\ell.$$
We can use the Euler characteristic formula for local Galois cohomology \cite{serre} 5.7, Theorem 5 together
with the fact that 
\beg{egal1}{V_\ell(E)^{Gal(\Q_\ell)}=0,} 
and that 
\beg{egal2}{Hom_{\Q_\ell}(V_\ell(E),\Q_\ell)(1)\cong V_\ell(E)}
as Galois representations.

(To prove \rref{egal1}, note that the same claim holds for every extension of $\Q_p$, and we will prove it in this case. By semi-stable reduction we may reduce to the cases, by taking a finite extension of the base field, if it is necessary, when $E$ has either good or split multiplicative reduction. In the first case by p-adic Hodge theory the claim fails then the rigid cohomology of degree 1 of the reduction of $E \mod p$ has a trivial factor as an F-isocrystal. This is not possible by the Weil conjectures \cite{delignew}. In the other case one can use the Tate uniformisation (cf. \cite{silver2} Section V.2) of $E$ to describe the Tate module of $E$ as a nontrivial extension of $\Q_p(1)$ by $\Q_p$, and hence there are no invariants in this case either.)

The Euler characteristic formula is stated for finite modules in \cite{serre},  but one can look at 
$T_\ell(E)/(\ell^m)$ and pass to the limit to get a statement about ranks. For any continuous 
$\Z_\ell[Gal(\Q_p)]$-module $T$ satisfying 
\rref{egal1} and \rref{egal2}, $V=T\otimes_{\Z_p}(\Q_p)$, 
the rank of $H^1(\Q_p,V)$ over $\Q_\ell$ is $0$ for $\ell\neq p$ and is equal 
to $rank_{\Q_p}(V)$ for $\ell=p$ and $0$ otherwise.
\end{proof}

\vspace{3mm}
 Now the inverse limit of the short exact sequences
\beg{eredd}{\diagram 0\rto & E_{\ell^k} \rto &  E \rto^{\ell^k} & E\rto &0
\enddiagram}
has the form
\beg{esolenoid}{0\r T_\ell(E)\r \lim_\leftarrow E \r E\r 0.
}
(Again, there is no $\lim^1$ since the maps are onto.)
Observe now that the first map factors as follows:
\beg{esolfac}{\diagram
T_\ell(E)\rrto\drto && \displaystyle\lim_\leftarrow E\\
& T_\ell(E)\otimes_{\Z_\ell}\Q_\ell\urdotted|>\tip
\enddiagram
}
Therefore, Lemma \ref{l2} has the following 

\begin{corollary}
\label{c1}
The connecting map of \rref{esolenoid}
\beg{econnect}{\delta=\delta_p:E(\Q_p)
\r H^{1}(\Q_p,T_\ell(E))
}
 is onto for $p\neq \ell$ and for $p=\ell$, if $xp\in Im(\delta)$ with $x\in H^1(\Q_p,T_p(E))$,
then $x\in Im(\delta)$.
\end{corollary}

\begin{proof}
For $p=\ell$, if $x\notin Im(\delta)$, then $x$ maps to a non-zero element of $H^1(\Q_p,\displaystyle \lim_\leftarrow E)$,
so $xp$ maps to a non-zero element of $H^1(\Q_p,\displaystyle \lim_\leftarrow E)$, contradicting $xp\in Im(\delta)$.
\end{proof}

We also note that $Coker(\delta_\ell)$ is equal to $T_\ell H^1(\Q_\ell,E)$, which is torsion-free (in fact,
isomorphic to $\Z_\ell$).

\vspace{3mm}
Now by our earlier discussion of the Hochschild-Serre
spectral sequence \rref{e3}, (e) is equivalent to the statement that
\beg{econnectq}{\delta_\Q:E(\Q)\otimes \Z_p\r H^1_{\text{{\em\'{e}t}}}(\Z[S^{-1}],T_p(E))
}
(which is injective) is onto, and to the statement of
Theorem \ref{t00}. 

\begin{lemma}
\label{lpt}
The inclusion of a $\ell$-decomposition subgroup in $Gal(\overline{\Q}/\Q)$ induces a homomorphism
\beg{egal3}{H^1_{\text{{\em\'{e}t}}}(\Z[S^{-1}],V_\ell(E))\r H^1(\Q_\ell,V_\ell(E))}
whose image is isomorphic to $\Q_\ell$.
\end{lemma}

\begin{proof}
Again, it is true more generally for any continuous $\Z_\ell[Gal(K/\Q)]$-module $T$,
$V=T\otimes_{\Z_p}\Q_p$,  satisfying \rref{egal1} and \rref{egal2}
that for $p=\ell$, \rref{egal3} is an inclusion whose image has $\Q_p$-rank equal to $rank_{\Q_p}(V)/2$. The Poitou-Tate exact sequence \cite{tate,mpt} is usually stated for a finite continuous $Gal(K/\Q)$-module $M$:
\beg{ept}{\resizebox{\textwidth}{!}{\diagram
0\rto & H^0_{\text{{\em\'{e}t}}}(\Z[S^{-1}],M)\rto &\prod_{p\in S} 
H^0(\Q_p,M)\rto & H^2_{\text{{\em\'{e}t}}}(\Z[S^{-1}],M^\prime)^*\dto&\\
&H^1_{\text{{\em\'{e}t}}}(\Z[S^{-1}],M^\prime)^*\dto & \prod^{\prime}_{p\in S} 
H^1(\Q_p,M)\lto & H^1_{\text{{\em\'{e}t}}}(\Z[S^{-1}],M)\lto&\\
&H^2_{\text{{\em\'{e}t}}}(\Z[S^{-1}],M)\rto &\bigoplus_{p\in S}
 H^2(\Q_p,M)\rto & H^0_{\text{{\em\'{e}t}}}(\Z[S^{-1}],M^\prime)^*\rto & 0
\enddiagram}
}
where $M^\prime$ denotes the Pontrjagin dual and $\prod^\prime$ the restricted product (of course, 
in the present case, they are the same thing, since $S$ is finite). Our statement can be proved
by replacing $T$ with $T/(\ell^m)$ and passing to the limit, also considering the fact that the Tate module is self-dual.
\end{proof}

Since  the groups 
$$H^0_{\text{{\em\'{e}t}}}(\Z[S^{-1}],E_{\ell^k})$$ 
are finite, 
there is no $\lim{}^1$ term, 
and we have
$$H^1_{\text{{\em\'{e}t}}}(\Z[S^{-1}],T_\ell(E))=\lim_\leftarrow 
H^1_{\text{{\em\'{e}t}}}(\Z[S^{-1}],E_{\ell^k}).$$
Therefore, for \rref{econnectq}, it suffices to prove that the connecting map of \rref{eredd}
\beg{eredonto}{E(\Q)\r Im(H^1_{\text{{\em\'{e}t}}}(\Z[S^{-1}],T_\ell(E))\r 
H^1_{\text{{\em\'{e}t}}}(\Z[S^{-1}],E_{\ell^k}))
}
is onto. Take an element $x\in H^1_{\text{{\em\'{e}t}}}(\Z[S^{-1}],T_\ell(E))$. By Corollary \ref{c1}, the image
of $x$ in $H^1(\Q_\ell,T_\ell(E))$ is in the image of the connecting map \rref{econnect}
if and only if this is true rationally, i.e. if
\beg{egal+}{
\parbox{3.5in}{The image of
every element $x\in H^1_{\text{{\em\'{e}t}}}(\Z[S^{-1}],V_\ell(E))$ 
in $H^1(\Q_\ell,V_\ell(E))$ is in the image of the connecting map
$$\delta:E(\Q_\ell)\widehat{\otimes}\Q_\ell\r H^1(\Q_\ell,V_\ell(E)).$$}
}
If $rank_\Q(E)>0$, then certainly \rref{egal+} is true: By Lemma \ref{lpt}, the image of 
$H^1_{\text{{\em\'{e}t}}}(\Z[S^{-1}],V_\ell(E))$ in
$H^1(\Q_\ell,V_\ell(E))$ is
a $1$-dimensional (over $\Q_\ell$) subspace of $H^1(\Q_\ell,V_\ell(E))$ and if 
$$rank_\Q(E)>0,$$ 
then a 
$1$-dimensional $\Q_p$-subspace of the target of $\delta$ in \rref{egal+} comes from the connecting map
\rref{econnectq}.

If \rref{egal+} holds, noting that $p=\ell$,
then the reduction $x_m$ of our element $x$ $\mod \ell^m$ is in the image of the connecting map.
This means that 
\beg{eselmerx}{x_m\in S^{(\ell^m)}(E/\Q).}
Therefore, assuming $\Sha$ has no $\ell$-divisibility beyond $\ell^s$, 
we see from \rref{eselmerx} that 
$x_k=x_{k+s+1}\mod\ell^k$ is in the image of the connecting map \rref{econnectq},
as claimed. 

(Note that, in general, an element $x\in H^1_{\text{{\em\'{e}t}}}(\Z[S^{-1}],T_\ell(E))$ lies
in the $\ell$-adic Selmer group $S_\ell=\displaystyle\lim_{\displaystyle \leftarrow} S^{(\ell^m)}$ if and only if
$x_\ell\notin Im(\delta_\ell)$ (provided $\ell\neq 2$). Furthermore, $S_\ell/Im(\delta_\Q)=T_\ell\Sha(E/\Q)$,
which means that $S_\ell=Im(\delta_\Q)$ if and only if $\Sha(E/\Q)\otimes \Z_{(\ell)}$ is finite.)

This completes the proof that (a) implies (e).

\vspace{3mm}

To prove that (e) implies that $rank_\Q(E)>0$, 
note that we can assume $\Sha(E)_{(\ell)}<\infty$, since we already proved that (e) implies it.
Therefore, we have an isomorphism
\beg{eshaq}{\diagram
E(\Q)\otimes \Q_\ell\rto^(.4)\cong &H^1_{\text{{\em\'{e}t}}}(\Z[S^{-1}],V_\ell E).
\enddiagram
}
On the other hand, the right hand side of \rref{eshaq} has non-zero $\Q_\ell$-rank, since
by Lemma \ref{lpt}, its image in \rref{egal3} has $\Q_\ell$-rank $1$.

\vspace{3mm}
The following proof that (f) is equivalent to (b) was pointed out to me by C.Weibel: We have
\beg{eh21}{H^2_{\text{{\em\'{e}t}}}(E(\Z[S^{-1}]),\Z/\ell^k(1))=Pic(E)/\ell^k,}
\beg{ehet21}{H^2_{\text{\'e}t}(E(\Z[S^{-1}]),\Z/\ell^k(1))= Pic(E)/\ell^k\oplus {}_{\ell^k}Br(E(\Z[S^{-1}])}
(by the long exact sequence in cohomology associated with the short exact sequence of sheaves \rref{egm}).
Also, we have
$$\begin{array}{l}H^{1}_{\text{{\em\'{e}t}}}(E(\Z[S^{-1}]),\Z/\ell^k(1))=\\
\Q^*/\Q^{*\ell^k}\oplus {}_{\ell^k}Pic(E)\cong H^{1}_{\text{{\em\'{e}t}}}(E(\Z[S^{-1}]),\Z/\ell^k(1)).
\end{array}$$
Now consider the diagram
$$\resizebox{5in}{!}{\diagram
0\rto &\displaystyle\lim_\leftarrow{}^{1}H^1(E(\Z[S^{-1}],\Z/\ell^k(1))\dto^\cong \rto &
H^2(E(\Z[S^{-1}]),\mathbf{Z}_\ell(1))\dto\rto &
\displaystyle\lim_\leftarrow H^2_{\text{{\em\'{e}t}}}(E(\Z[S^{-1}]),\Z/\ell^k(1))\rto \dto & 0\\
0\rto &
\displaystyle\lim_\leftarrow{}^{1}H^1_{\text{\'{e}t}}(E(\Z[S^{-1}]),\Z/\ell^k(1))\rto &
H^2_{\text{\'{e}t}}(E(\Z[S^{-1}]),\mathbf{Z}_\ell(1))\rto &
\displaystyle\lim_\leftarrow H^2_{\text{\'{e}t}}(E(\Z[S^{-1}]),\Z/\ell^k(1)) \rto& 0.
\enddiagram}
$$
Using \rref{eh21} and \rref{ehet21}, we obtain 
$$\begin{array}{l}Coker(H^2(E(\Z[S^{-1}]),\mathbf{Z}_\ell(1))\r H^2_{\text{\'{e}t}}(E(\Z[S^{-1})
,\mathbf{Z}_\ell(1)))\\=T_\ell Br(E(\Z[S^{-1}])).\end{array}$$
Of course, one can make the same calculation with $E$ replaced by $Spec(\Q)$, and naturality then
gives that (f) is equivalent to (b).
\qed

\vspace{3mm}
Finally, let us see why these arguments do not work with 
$$H^1_{\text{{\em\'{e}t}}}(\Z[S^{-1}],T_p(E))$$
replaced by the global Galois cohomology group $H^1(\Q,T_p(E))$. First, we note that
\rref{ept} in fact holds for a finite module $M$ without assuming that $S$ is finite. Thus, we may
replace $S$ by the set of all primes in $\Q$, which amounts to replacing $\Z[S^{-1}]$ by $\Q$.
Going even further, we see that the relevant $\lim_\leftarrow{}^{1}$-terms in this case are also
$0$, so even \rref{ept} remains valid with $\Z[S^{-1}]$ replaced by
$\Q$. 

The problem, however, is that the cohomology group $H^1(\Q,T_p(E))$ is huge: 
for $S$ equal to the set of all primes in $\Q$, $M=T_\ell(E)/(\ell^k)$, we can find 
infinitely many primes $p$ for which $E(\Q_p)$ has $\ell^k$-torsion: By Chebotarev density theorem,
it suffices to choose primes $p$ at which $T_\ell(E)$ is unramified, and for which the Frobenius acts
trivially on the field obtained by attaching the $\ell^k$-$E$-torsion to $\Q$. In the inverse limit 
of the middle term \rref{ept} in the case when $S$ contains all primes in $\Q$, then, the inverse limit
of these infinite products of $\ell^k$-torsion modules over diminishing sets
of primes will create an uncountable non-torsion submodule, which must come from an uncountable torsion
submodule of $H^1(\Q,T_\ell(E))$ by the exactness of the limit of \rref{ept}. In view of the
above discussion, this also proves
Theorem \ref{tm1}. We remark that this argument is similar to the method of Kolyvagin \cite{kol1}.

\vspace{3mm}

\section{An Example}\label{s5}

In this Section, we give an example where we can show for an elliptic curve that \rref{egal+}
is false directly. In the example, the conclusion can be made
by Galois cohomology computations, thus in particular showing directly that $rank_\Q(E)=0$ for any elliptic 
curve over $\Q$ with the same Galois data.

The elliptic curve over $\Q$ given by the equation
$$y^2=x^3+x^2-117x-541$$
has conductor $N=2^4\cdot 11^2$, CM, ($0$ rank), torsion $1$ and ordinary good reduction at $3$
(see the Cremona tables at 

\noindent http://johncremona.github.io/ecdata/).
Looking at the smallest number field $K$ over which $E$ has $\Z/9\times \Z/9$ torsion, one sees that
$$G=Gal(K/\Q)\cong \Z/6\times \Sigma_3.$$
The $3$-decomposition subgroup of $G$ is the normal subgroup 
$$G_3=\Z/6\times \Z/3.$$
The representation of $G$ on $T_3(E)/(9)$ can be described (up to isomorphism) as follows: the generator of the
$\Z/6$-factor acts by 
$$\left(
\begin{array}{rr}
2 &0\\0& 2
\end{array}
\right),$$
the generator of the $\Z/3$-subgroup of $\Sigma_3$ acts by
$$\left(
\begin{array}{rr}
4 &0\\0& 7
\end{array}
\right),$$
the generator $\sigma$ of a $\Z/2$-subgroup of $\Sigma_3$ acts by
$$\left(
\begin{array}{rr}
0&1\\1 &0
\end{array}
\right).$$
(The author obtained these results using SAGE.)
Let $\Gamma=Gal(\Q)$, and let $\Gamma_{(3)}\cong Gal(\Q_3)$ be a $3$-decomposition subgroup. Let
also $\Gamma_3$ be the pullback of $\Gamma$ by $G_3\r G$. Then we have inclusions
$$\Gamma_{(3)}\subset \Gamma_3\subset \Gamma.$$

\begin{lemma}
\label{lex}
The restriction 
$$H^1(\Gamma_3,T_3(E)/(9))\r H^1(\Gamma_{(3)},T_3(E)/(9))\cong
\Z/9\times \Z/9$$
is an isomorphism.
\end{lemma}

\begin{proof}
Once again, we know that the group $H^1(\Gamma_{(3)},T_3(E)/(9))$ is isomorphic to 
$\Z/9\times \Z/9$ by the Euler characteristic formula (\cite{serre}, 5.7, Theorem 5). The corresponding cohomology 
group of 
$\Gamma_3$ can be computed in the following way: let $Q$ be the $\Gamma$-module coinduced from the
$\Gamma_3$-module $T_3(E)/(9)$. Then $|Q|=9^4$, so by the Poitou-Tate exact sequence \rref{ept},
$$H^1(\Gamma,Q)\cong H^1(\Gamma_3, T_3(E)/(9))\cong (\Z/9)^2.$$
To see that the restriction is an isomorphism, pick a $G_3$-equivariant section
$$\diagram
T_3(E)/(9)\drto_{Id}\rto & Q\dto\\
& T_3(E)/(9)
\enddiagram$$
and consider the commutative diagram
$$
\diagram
H^1(\Gamma,Q)\dto_\cong\rto & H^1(\Gamma_{(3)},Q)\\
H^1(\Gamma_3,T_3(E)/(9))\rto & H^1(\Gamma_{(3)}, T_3(E)/(9))\uto
\enddiagram
$$
and the Poitou-Tate exact sequence \rref{ept}.
\end{proof}

Note that $T_3(E)/(9)$ as a $G_3$-module actually splits as
$$T_3(E)/(9)\cong M\oplus M^\prime$$
where both $M, M^\prime$, as abelian groups, are isomorphic to $\Z/9$.
The first cohomologies of both $\Gamma_{(3)}$ and $\Gamma_3$ on $M$ and $M^\prime$
are therefore isomorphic to $\Z/9$. 

\vspace{3mm}
Now the image of 
$$\Z/9\cong H^1(\Gamma,T_3(E)/(9))\subset H^1(\Gamma_3,T_3(E)/(9))\cong \Z/9\oplus \Z/9$$
is, by the Hochschild-Serre spectral sequence, invariant under the action of the involution $\sigma$.
We see that neither of the subgroups $H^1(\Gamma_3,M)$ nor $H^1(\Gamma_3,M^\prime)$ satisfies
this property, since $\sigma$ switches them.

This describes an obstruction modulo $3$, but it lifts to a non-torsion obstruction. In fact, Kato \cite{kato} in 
Chapter 16 relates $L_{p}$ to the $p$-adic $L$-function in case of newforms with ordinary good reduction.
In our case, the Galois cohomology calculation shows that the $3$-adic $L$-function is non-zero $\mod 3$. 
It is interesting, however, that the calculation is elementary, giving hope that non-vanishing of the obstruction can
be shown in contexts where modularity is not known, for example over more general number fields or
for abelian varieties.

\end{document}